\documentclass[11pt]{article}
 \usepackage{bbm}
\usepackage{amssymb, amsthm, amsmath, amscd}
\newtheorem{thm}{Theorem}[section]
\newtheorem{lem}[thm]{Lemma}
\newtheorem{prop}[thm]{Proposition}
\newtheorem{cor}[thm]{Corollary}
\newtheorem{NN}[thm]{}
\theoremstyle{definition}\newtheorem{df}[thm]{Definition}
\theoremstyle{definition}\newtheorem{rem}[thm]{Remark}
\theoremstyle{definition}

\renewcommand{\phi}{\varphi}

\newcommand{\N}{\mathbb{N}}
\newcommand{\Z}{\mathbb{Z}}
\newcommand{\Q}{\mathbb{Q}}
\newcommand{\R}{\mathbb{R}}
\newcommand{\C}{\mathbb{C}}
\newcommand{\T}{\mathbb{T}}

\newcommand{\Aff}{\operatorname{Aff}}

\newcommand{\morp}{contractive completely positive linear map}

\newcommand{\hm}{homomorphism}
\newcommand{\dt}{\delta}
\newcommand{\ep}{\epsilon}

\newcommand{\cd}{\cdots}

\newcommand{\andeqn}{\,\,\,{\rm and}\,\,\,}
\newcommand{\rforal}{\,\,\,{\rm for\,\,\,all}\,\,\,}
\newcommand{\CA}{$C^*$-algebra}

\newcommand{\SCA}{$C^*$-subalgebra}

\newcommand{\af}{{\alpha}}
\newcommand{\bt}{{\beta}}

\newcommand{\beq}{\begin{eqnarray}}
\newcommand{\eneq}{\end{eqnarray}}
\newcommand{\tforal}{\,\,\,\text{for\,\,\,all}\,\,\,}
\newcommand{\tand}{\,\,\,\text{and}\,\,\,}
\newcommand{\wtlog}{without loss of generality}

\usepackage{amsfonts}
\usepackage{mathrsfs}
\usepackage{textcomp}
\usepackage[all]{xy}

\title{ Crossed products and minimal dynamical systems}
\author{Huaxin Lin
 }
\date{}

\begin{document}

\maketitle

\begin{abstract}
Let $X$ be an infinite compact metric space with finite covering dimension and let $\af, \bt : X\to X$ be two minimal 
homeomorphisms. We prove that the crossed product \CA s $C(X)\rtimes_\af\Z$  and $C(X)\rtimes_\bt\Z$ are 
isomorphic if and only if they have isomorphic Elliott invariant.  In a more general setting, we show 
that if $X$ is an infinite compact metric space and if $\af: X\to X$ is a minimal homeomorphism 
such that $(X, \af)$ has mean dimension zero, then the tensor product of the crossed product with a UHF-algebra
of infinite type has generalized tracial rank at most one. This implies that the crossed product 
is in a classifiable class of amenable simple \CA s.
 \end{abstract}
 
 \section{Introduction}
 
 Let $X$ be an infinite  compact metric space and let $\af: X\to X$ be a minimal homeomorphism.
 It is well known that the crossed product \CA\, $C(X)\rtimes_\af\Z$ is a unital simple separable amenable 
 \CA\, which satisfies the Universal Coefficient Theorem (UCT).  There has been a great deal of interaction between 
 dynamical systems and \CA\, theory.  Let $\bt: X\to X$ be another minimal homeomorphism.
 A result of Tomiyama (\cite{To}) stated that two dynamical systems $(X, \af)$ and $(X, \bt)$ are flip conjugate if and only 
 if  the crossed products $C(X)\rtimes_\af\Z$ and $C(X)\rtimes_\bt Z$ are isomorphic preserving $C(X).$ 
 When $X$ is the Cantor set,  Giordano, Putnam  and Skau (\cite{GPS}) proved that two such systems are strong 
 orbit equivalent if and only if the crossed products are isomorphic as \CA s.  In this case, 
 the crossed products are isomorphic to a unital simple $A\T$-algebra of real rank zero. 
 The Elliott program of classification of amenable \CA s plays  an important role in this case.  In fact, 
 $K$-theory can be used to determine when two minimal Cantor systems are strong orbit equivalent. 
 
 The irrational rotation algebras may be viewed as crossed products from the minimal dynamical systems 
 on the circle by irrational rotations. Elliott and Evans (\cite{EE}) proved a structure theorem which states that
 every irrational rotation \CA\, is isomorphic to a unital simple $A\T$-algebra of real rank zero (there 
 are also some earlier results such as \cite{CEY}).   With the rapid development in the Elliott program, 
 it becomes increasingly important to answer the question when  crossed product \CA s from
  minimal dynamical systems are classifiable. If $A$ is a unital separable simple \CA\, of stable rank one
 with the tracial state space $T(A),$ then the image of $K_0(A)$ under $\rho_A$ in $\Aff(T(A))$ is dense if $A$ has real rank zero
 (\cite{BH}), where $\rho_A: K_0(A)\to \Aff(T(A))$ is defined by 
 $\rho_A([p])=\tau(p)$ for all projection $p\in A$ and for all $\tau\in T(A)$ (see \ref{Drho} below).
 Let $X$ be an infinite  compact metric space with finite covering dimension and let $\af: X\to X$ be a minimal homeomorphism. 
 In \cite{LP}, it shown that 
  $C=C(X)\rtimes_\af\Z$ has tracial rank zero if and only 
 if the image of $K_0(C)$  under $\rho_C$ is dense in $\Aff(T(C)).$ 
 If $X$ is connected and $\af$ is uniquely ergodic, and if 
 the rotation number of $(X, \af)$ has an irrational value, then the image of $K_0(C)$ is dense in 
 $\Aff(T(C)).$   This recovers the earlier result of Elliott and Evans \cite{EE}.  This is because that 
 unital separable simple amenable \CA s of tracial rank zero in the UCT class can be classified by the Elliott invariant up to isomorphism (\cite{Lnduke}).
 The Elliott program  recently moved beyond the \CA s of finite tracial rank.  With Winter's method (see \cite{WI}),
 simple \CA s with finite rational tracial rank in the UCT class have been classified by their 
 Elliott invariant (see \cite{Lncrell}, \cite{LNlift} and \cite{Lninv}). Toms and Winter (\cite{TW}) improved the result in \cite{LP} by showing that,  if projections 
 in $C(X)\rtimes_\af\Z$ separate the tracial state space (the set of $\af$-invariant Borel probability measures), then 
 the crossed products have rational tracial rank zero and therefore are classifiable by Elliott invariant. Suppose 
 that $\af,\,\bt: X\to X$ are two minimal homeomorphisms and both are uniquely ergodic.
 Then $C(X)\rtimes_\af\Z$ is isomorphic to $C(X)\rtimes_\bt \Z$ if and only if their $K$-groups 
 are isomorphic in a way which also preserves the order and order unit of their $K_0$-groups. 
 However,  as early as the 1960's, Furstenburg (\cite{Fur}) presented minimal homeomorphisms 
 on the 2-torus which are not Lipschitz. The set of projections of  the crossed product \CA s associated with the minimal 
 dynamical systems  in this case could not separate the tracial state space of the crossed products. 
 The author was asked whether crossed products of Furtenburg transformation on 2-torus which are not Lipschitz 
 are $A\T$-algebras. 
 These crossed products have the same Elliott invariant  as those of unital simple \CA s with tracial rank one (but not zero). 
 They do not have rational tracial rank zero. On the other hand further developments were made by K. Strung (\cite{Sk}) who showed that the crossed product \CA s from certain minimal homeomorphisms 
 on odd dimensional spheres have rational tracial rank one  and therefore are classifiable. In these cases, projections in the crossed products may not separate the tracial state space. 
 It   has lately been  shown that the crossed product \CA s from any minimal homeomorphisms on  
 $2d+1$ spheres (with $d\ge 1$)  and other odd dimensional spaces have rational tracial rank at most one (\cite{Lncjm15}).  These results, however,   do not cover 
 the cases of the 2-torus given by Furstenburg.  In fact,  there are  many minimal dynamical systems  on connected 
 spaces with complicated 
   simplexes of invariant probability Borel measures.  In this note we show the following:

 \begin{thm}\label{MTA}
 Let $X$ be an infinite compact metric space  with finite covering dimension and let $\af: X\to X$ be a minimal homeomorphism.
 Then $C(X)\rtimes_\af\Z$ belongs  to ${\cal N}_1^{\cal Z},$ a classifiable class of \CA s (see \ref{DN1} below).
 If $Y$ is another compact metric space and let $\bt: Y\to Y$ be another minimal homeomorphism.
 Then $C(X)\rtimes_\af\Z\cong C(Y)\rtimes_\bt \Z$ if and only 
 if ${\rm Ell}(C(X)\rtimes)_\af\Z)\cong {\rm Ell}(C(Y)\rtimes_\bt \Z).$ 
 \end{thm}
 \vspace{-0.1in}
 \noindent(see  \ref{DEll}  for the definition  ${\rm Ell}(\cdot)$ below).
 
 We actually prove the following:
 
 \begin{thm}\label{MTB}
 Let $X$ be an infinite  compact metric space and let $\af: X\to X$ be a minimal homeomorphism
 such that $(X, \af)$ has mean dimension zero. 
 Then $C(X)\rtimes_\af\Z$ belongs  to ${\cal N}_1^{\cal Z}.$
 \end{thm}

It should noted that if $X$ has finite dimension, then the minimal dynamical system $(X, \af)$ always has mean dimension zero.  So Theorem \ref{MTA} follows from \ref{MTB}.
 Moreover, if the minimal dynamical system $(X, \af)$ has countably many 
extreme $\af$-invariant Borel probability measures, then $(X, \af)$ has mean dimension zero.  
In fact, we show that $(C(X)\rtimes_\af\Z)\otimes {\cal Z}$ is always classifiable, where ${\cal Z}$ is the Jiang-Su algebra.
It is proved in \cite{EN} that crossed product \CA s from minimal dynamical systems of mean dimension zero 
are ${\cal Z}$-stable.

 Let $X$ be a compact manifold and $\af: X\to X$ be a minimal  diffeomorphism.  In  a long  paper,  Q. Lin and N.C. Phillips (\cite{qLP})
showed that $C(X)\rtimes_\af\Z$ is an inductive limit of recursive sub-homogenous \CA s (with bounded dimension in their 
spectrum). With the classification result in \cite{GLN}, we offer the following:

\begin{cor}\label{MCC}
 Let $X$ be an infinite  compact metric space and let $\af: X\to X$ be a minimal homeomorphism
 such that $(X, \af)$ has mean dimension zero. Then $C(X)\rtimes_\af\Z$ is an inductive limit 
 of sub-homogenous \CA s described in {\rm \ref{DMtorus}} with dimension of the spectrum at most three. 
\end{cor}

The proof of these results is based on recent advances in the Elliott program. In \cite{GLN}, it is shown 
that the class ${\cal N}_1^{\cal Z}$ of unital separable simple ${\cal Z}$-stable \CA s which have rational generalized tracial rank 
at most one in the UCT class can be classified by the Elliott invariant.  
It should be noted that  unital infinite dimensional simple 
\CA s that have generalized tracial rank at most one  are all ${\cal Z}$-stable. The range theorem in \cite{GLN} shows 
that the class ${\cal N}_1^{\cal Z}$ exhausts the Elliott invariants of all possible unital simple amenable ${\cal Z}$-stable 
\CA s and the range of the Elliott invariants is characterized. Moreover, it also shows that a \CA\, $A\in {\cal N}_1^{\cal Z}$ 
is isomorphic to a  unital simple inductive limit of subhomogenous algebras of a certain  special form (see \ref{DMtorus} below)
such that the spectrum has dimension at most three.  This implies that $C(X)\rtimes_\af\Z$ has the same Elliott invariant as one of the model \CA s in \cite{GLN} as long as it is ${\cal Z}$-stable.
This immediately provides an opportunity for $C(X)\rtimes_\af\Z$ to be embedded into a model \CA s as presented in \cite{GLN}. 
Indeed an earlier proof of the main result of this note did just that. However, 
this is unnecessarily difficult since it involves  exact embedding which requires usage of an asymptotic unitary equivalence 
theorem and repeated usage of a version of Basic Homotopy Lemma. It also uses a new characterization of 
TA${\cal C}$ for some special class ${\cal C}$ of weakly semiprojective \CA s. 
We realized later that it is not necessary to prove the actual embedding. An approximate version of it would be 
sufficient. This greatly simplifies the proof. 

It is worth point out that, without referring to the more recent long paper \cite{GLN},  but referring 
to \cite{Lninv} and  \cite{LNjfa}, the proof we used in this paper 
still gives the following result: 
\begin{thm}\label{Trtr1}
Let $X$ be an infinite  compact metric space and let $\af: X\to X$ be a minimal homeomorphism
 Then $C(X)\rtimes_\af\Z$ has rationally tracial rank at most one  if 
and only if $K_0(C(X)\rtimes \Z)\otimes \Q$ is a dimension group and
each extremal trace gives an extremal state 
on $K_0(C(X)\rtimes \Z.$ 
 \end{thm}
 
  This result, of course, is a corollary of \ref{MTB}. But it has already covered significant large class of crossed from minimal action on 
compact metric spaces. For example, this result implies that the crossed products from non-Lipschitz Furstenburg transformations
on 2-torus are unital simple $A\T$-algebras.  Theorem \ref{Trtr1} also covers all minimal actions 
on odd dimensional spaces  studied in \cite{Lncjm15}
(see \ref{CCC} below and its remark).

This note is organized as follows. The next section serves as  a preliminary for the later sections. In Section 3, using  a 
recent characterization of $TA{\cal C}$ for unital simple \CA s with finite nuclear dimension (\cite{W13}), we present another convenient 
characterization for a unital separable simple amenable \CA\, to have rational generalized tracial rank at most one.
In Section 4, we present a few refinements of results in \cite{GLN} to be used in this note. In Section 5, we present 
the proof for the main results presented earlier in the introduction. 

\vspace{0.2in}

The main results  of this note were reported at  ``Dynamics and C*-algebras: amenability and soficity", 
a workshop at BIRS Banff Research Station in October 2014. 

\vspace{0.1in}

{\bf Acknowledgement}: This work was partially supported by  The Research Center for Operator Algebras 
at East China Normal University and by an NSF grant.  The author would like to thank Z. Niu who informed the 
the preprint \cite{ENSTW}.

  \section{Preliminaries}

\begin{df}\label{Du}
{\rm
Let $A$ be a unital \CA. Denote by $U(A)$ the unitary group of $A$ and $U_0(A)$ the normal subgroup
of $U(A)$ consisting of the path connected component containing $1_A.$  Denote by
$CU(A)$ the closure of the commutator subgroup of $U_0(A).$
The map $u\mapsto {\rm diag}(u, 1_A)$ gives a \hm\, from $U(M_n(A))$ to $U(M_{n+1}(A))$ for each integer $n\ge 1.$
We write $U(M_{\infty}(A))$ for $\cup_{n=1}^{\infty}U(M_n(A))$ by using the above inclusion.
Note we also use the notation  $U_0(M_{\infty}(A))=\cup_{n=1}^{\infty}U_0(M_n(A))$ and 
$CU(M_{\infty}(A))=\cup_{n=1}^{\infty}CU(M_n(A)).$ 
}
\end{df}
\begin{df}\label{Aq}
{\rm Let $A$ be a unital \CA\, and let $T(A)$ be the tracial state space.
Let $\tau\in T(A).$ We say that $\tau$ is faithful if $\tau(a)>0$ for all $a\in A_+\setminus\{0\}.$
Denote by $T_{\bf f}(A)$ the set of all faithful tracial states.

Denote by $\Aff(T(A))$ the space of all real continuous affine functions on $T(A)$ and 
denote by ${\rm LAff}_b(T(A))$ be the set of all bounded lower-semi-continuous real affine functions on $T(A).$ 

Suppose that $T(A)\not=\emptyset.$ There is an affine  map
$r_{aff}: A_{s.a.}\to \Aff(T(A))$ by
$$
r_{aff}(a)(\tau)=\hat{a}(\tau)=\tau(a)\tforal \tau\in T(A)
$$
and for all $a\in A_{s.a.}.$ Denote by $A_{s.a.}^q$ the image $r_{aff}(A_{s.a.})$ and
$A_+^q=r_{aff}(A_+).$

{For each integer $n\ge 1$ and $a\in M_n(A),$
write $\tau(a)=(\tau\otimes Tr)(a),$ where $Tr$ is the (non-normalized) trace on $M_n.$}
}

\end{df}

\begin{df}\label{Drho}
{Let $A$ be a  stably finite \CA\, with $T(A)\not=\emptyset.$ Denote by $\rho_A: K_0(A)\to \Aff(T(A))$ the  order preserving \hm\,  defined by $\rho_A([p])=\tau(p)$ for any projection $p\in M_n(A)$ (see the above convention), $n=1,2,....$}

{Let $A$ be a unital stably finite \CA. A map $s: K_0(A)\to \R$ is said to be a state if $s$ is an order preserving \hm\, such that
$s([1_A])=1.$ The set of states on $K_0(A)$ is denoted by $S_{[1_A]}(K_0(A)).$ }

{Denote by $r_A: T(A)\to S_{[1_A]}(K_0(A))$ the map defined by $r_A(\tau)([p])=\tau(p)$ for all projection 
$p\in M_n(A)$ (for any integer $n$) and for all $\tau\in T(A).$ }
\end{df}

{
\begin{df}\label{DEll}
Let $A$ and $C$ be two unital separable stably finite \CA\, with $T(C)\not=\emptyset$ and 
$T(A)\not=\emptyset.$ Let $\kappa_i: K_i(C)\to K_i(A)$ ($i=0,1$) be \hm\, and let 
$\lambda: T(A)\to T(C)$ be an affine continuous map.
We say that $(\kappa_0, \lambda)$ is compatible, if 
$$
r_C(\lambda(t))(x)=r_A(t)(\kappa_0(x))\rforal x\in K_0(C)\andeqn \rforal t\in T(A).
$$
Denote by $\lambda_{\sharp}: \Aff(T(C))\to \Aff(T(A))$ the induced affine continuous map defined by
$\lambda_{\sharp}(f)(\tau)=f\circ \lambda(\tau)$ for all $f\in \Aff(T(A))$ and $\tau\in T(C).$
When $(\kappa_0, \lambda)$ is compatible, denote by ${\overline{\gamma_{\sharp}}}: 
\Aff(T(C))/\overline{\rho_A(K_0(C))}\to \Aff(T(A))\overline{\rho_A(K_0(A))}$ the induced continuous map. 

Let $A$ be a unital simple \CA. The Elliott invariant of $A$, denote by ${\rm Ell}(A)$ is the following
six tuple
$$
{\rm Ell}(A)=(K_0(A), K_0(A)_+, [1_A], K_1(A), T(A), r_A).
$$
Suppose that $B$ is another unital simple \CA. We write  ${\rm Ell}(A)\cong {\rm Ell}(B),$ if
there is an order isomorphism $\kappa_0: K_0(A)\to K_0(B)$
such that $\kappa_0([1_A])=[1_B],$ an isomorphism $\kappa_1:
K_1(A)\to K_1(B)$ and an affine homeomorphism
$\kappa_\rho: T(B)\to T(A)$ such that $(\kappa_0, \kappa_\rho)$ is compatible.
\end{df}
}

{
\begin{df}\label{DW(A)}
Let $A$ be a \CA. Let $a, b\in M_n(A)_+$. Following Cuntz, we write
$a\lesssim b$ if there exists a sequence $(x_n)\subset M_n(A)$ such that
$\lim_{n\to\infty} x_n^*bx_n=a.$ If $a\lesssim b$ and $b\lesssim a,$ then
we write $a\sim b.$ The relation ``$\sim$" is an equivalence relation.
Denote by $W(A)$ the Cuntz semi-group of the equivalence classes of positive
elements in $\cup_{m=1}^{\infty} M_m(A)$ with orthogonal
addition and order ``$\lesssim$". 
In particular,  when $A$ has stable rank one, if $p, q\in M_n(A)$ are two projections then $p\sim q$ if and only if they are von Neumann equivalent. 
Denote by $\text{Cu}(A)$ the Cuntz semigroup $W(A\otimes {\cal K})$ with order ``$\lesssim$". 
\end{df}
}


\begin{df}\label{Aq1}
Let $A$ be a \CA.  Denote by $A^{\bf 1}$ the unit ball of $A.$
$A_+^{q, {\bf 1}}$ is the image of the intersection of $A_+\cap A^{{\bf 1}}$ in $A_+^q.$
\end{df}

\begin{df}\label{DinduceM}
Let $A$ and $B$ be two unital \CA s and let $\phi: A\to B$ be a  \hm. 
We write $\phi_{*i}: K_i(A)\to K_i(B)$ for the induced \hm s on $K$-theory and $[\phi]$ for 
the element in $KK(A,B)$ as well as $KL(A,B)$ if there is no confusion. 
We also use $\phi^{\ddag}: U(M_{\infty}(A))/CU(M_{\infty}(A))\to U(M_{\infty}(B))/CU(M_{\infty}(B))$
for the induced \hm. Suppose that $T(A)$ and $T(B)$ are both non-empty.
Then $\phi_T: T(B)\to T(A)$ is the affine continuous map 
induced by $\phi_T(\tau)(a)=\tau(\phi(a))$ for all $a\in A$ and $\tau\in T(A).$ 
Denote by $\phi^{cu}: \text{Cu}(A)\to \text{Cu}(B)$ the semigroup \hm s which preserves the order. 

We use $KL_e(A,B)^{++}$ for the subset of elements $\kappa\in KL(A,B)$ such that
$\kappa (K_0(A)\setminus\{0\})\subset K_0(B)\setminus \{0\}$ and $(\kappa([1_A])=[1_B].$
\end{df}

\begin{df}\label{DclassC}
Denote by ${\cal C}$ the class of those unital \CA s $C$ which are finite dimensional \CA s or those 
$C$ which are the pull-back:
\begin{equation}\label{pull-back}
\xymatrix{
C \ar@{-->}[rr] \ar@{-->}[d]^-{\pi_e}  && C([0,1], F_2) \ar[d]^-{(\pi_0, \pi_1)} \\
F_1 \ar[rr]^-{(\phi_0, \phi_1)} & & F_2 \oplus F_2,
}
\end{equation}
where $F_1$ and $F_2$ are finite dimensional \CA s and $\phi_i: F_2\to F_1$ are 
\hm s.  These \CA s are also called one dimensional non-commutative CW complexes (NCCW).
\CA\, $C$  can be also written as 
\beq\label{ETblock}
C=\{ (f,a)\in C([0,1], F_2)\oplus F_1: \phi_0(a)=f(0)\andeqn \phi_1(a)=f(1)\}
\eneq
and is also called Elliott-Thomsen building blocks.
Denote by ${\cal C}_0$ be the those \CA s $C$ in ${\cal C}$ with $K_1(C)=\{0\}.$ 

\CA s in ${\cal C}$ are semiprojective  (proved in  \cite{ELP}).
\end{df}

\begin{df}\label{DN1}
Let $A$ be a unital simple \CA\, and let ${\cal S}$ be a class of  unital \CA s. We say $A$ is $TA{\cal S},$ if 
for any $\ep>0,$ any finite subset ${\cal F}\subset A$ and any $a\in A_+\setminus \{0\},$ there exists 
a projection $p\in A$ and \SCA\, $C\in {\cal S}$ with $1_C=p$ such that
\beq\label{Dgtr-1}
&&\|px-xp\|<\ep \andeqn {\rm dist}(pxp, C)<\ep\rforal x\in {\cal F}\andeqn\\
&&1-p\lesssim a.
\eneq
In the case that ${\cal S}={\cal C},$ then we say $A$ has generalized tracial rank at most one and write 
 $gTR(A)\le 1.$
If $gTR(A)\le 1,$ we may say $A$ is $\text{TA}{\cal C}.$  
In the above definition, if $C\in {\cal C}_0,$ then we say $A$ is in $\text{TA}{\cal C}_0.$ 

It is proved in \cite{GLN} that, if $gTR(A\otimes Q)\le 1, $ where $Q$ is  the  UHF-algebra 
with $K_0(Q)=\Q,$ then $A\otimes Q\in \text{TA}{\cal C}_0$(Cor. 29.3 of \cite{GLN}).  By 
a result in \cite{LS}, this implies that $A\otimes U$ is $TA{\cal C}_0$ for all UHF-algebras $U$ of infinite 
type. 

Denote by ${\cal N}_1$ the class of unital simple amenable \CA s in the UCT class such that
$gTR(A\otimes U)\le 1$ for some UHF-algebra $U$ of infinite type. 

Denote by ${\cal Z}$ the Jiang-Su algebra of unital simple \CA\, with ${\rm Ell}({\cal Z})={\rm Ell}(\C)$
which is also an inductive limit sub-homogenous \CA s (\cite{JS}). Recall that a \CA\, $A$ is ${\cal Z}$-stable 
if $A\cong A\otimes {\cal Z}.$ 
Denote by ${\cal N}_1^{\cal Z}$ those \CA s in ${\cal N}_1$ which are ${\cal Z}$-stable.  
\end{df}

\begin{rem}\label{DMtorus}
~~Recall  \cite{EG}  that the finite CW complexes $T_{II,k}$  (or
$T_{III,k}$) is defined to be a 2-dimensional connected finite CW
complex with $H^2(T_{II,k})=\Z/k$ and
 $H^1(T_{II,k})=0$ (or 3-dimensional finite CW complex with $H^3(T_{III,k})=\Z/k$
 and $H^1(T_{III,k})=0=H^2(T_{III,k})$). 
For each n, there is a space $X_n'$ which is
 of form
$$X_n=[0,1]\vee S^1\vee S^1\vee\cd \vee S^1\vee T_{II,k_1} \vee T_{II,k_2} \vee \cd \vee T_{II,k_i} \vee T_{III,m_1} \vee T_{III,m_2}\vee \cd \vee T_{III,m_j}.$$
Let $P_n\in M_\infty (C(X_n))$  be a projection with rank $r(n,1).$ 
Let
$$
A_n=P_nM_{\infty}(C(X_n))P_n\bigoplus \bigoplus_{i=2}^{p_n} M_{r(n,i)}(\C)
$$
and $F_n=\bigoplus_{i=1}^{p_n}M_{r(n,i)}(\C).$
Fix a point $\xi\in X_n$ and denote by $\pi_\xi: A_n\to F_n$ by $\pi_\xi(f,a)=f(\xi)\oplus a,$ 
where $f\in P_nM_{\infty}(C(X_n))P_n$ and $a\in \oplus_{i=2}^{p_n}M_{r(n,i)}(\C).$ 
Let $E_n$ be a finite dimensional \CA\, and  let $\phi_0, \phi_1:F_n\to E_n$ be two unital \hm s.
Define 
$$
C_n=\{(c,a)\in C([0,1], E_n)\oplus A_n:  c(0)=\phi_0\circ \pi_\xi(a)\andeqn c(1)=\phi_1\circ \pi_\xi(a)\}.
$$

It is proved in \cite{GLN} (see Theorem 13.41 and Theorem 29.4 of \cite{GLN})  that every \CA\, in ${\cal N}_1^{\cal Z}$ can be written as 
simple inductive limit of \CA s of the form $C_n$ above (up to isomorphism). 
\end{rem}

\section{Approximate embeddings}

\begin{df}\label{DS}
Denote by ${\cal S}$ a class of unital \CA s which has the following properties:
(1) ${\cal S}$ contains all finite dimensional \CA s,
(2) tensor products of finite dimensional \CA s with   \CA s in ${\cal S}$  are in ${\cal S},$ 
(3) ${\cal S}$ is closed under direct sums,
(4) every \CA\, in ${\cal C}$ are weakly semiprojective,  and (5) if $S\in {\cal S},$ $J\subset {\cal S}$ is a closed two-sided 
ideal of $S,$ $\ep>0$ and ${\cal F}\subset S/I$ is a finite 
subset, then there exists \SCA\, $C\subset A/I$ such that $C\in {\cal S}$ and ${\rm dist}(x, C)<\ep\rforal 
x\in {\cal F}.$

It is proved in \cite{GLN} (see Lemma 3.20 in \cite{GLN}) that the class ${\cal C}$ in \ref{DclassC} satisfies (1)--(5). 
\end{df}

The following is proved in \cite{SW}.

\begin{prop}\label{LWS}{\rm{(cf. Lemma 2.1 \cite{SW})}}
Let ${\cal S}$ be a class of unital \CA s in \text{\ref{DS}} and let $U$ be a UHF-algebra of infinite type.
Let $A$ be a unital separable simple stably finite exact \CA.  Then $A\otimes U$ is $TA{\cal S}$ if and only if 
there is  $\eta>0$  such that, for any $\ep>0$ and any finite subset ${\cal F}\subset A\otimes U,$ there exists 
a projection $p$ and a \SCA\, $B\subset A\otimes U$ with $1_B=p$ and $B\in {\cal S}$  such that
\beq\label{LWS-1}
\tau(p)>\eta\rforal \tau\in T(A\otimes U),\\
\|px-xp\|<\ep\andeqn {\rm dist}(pap, B)<\ep\rforal x\in {\cal F}.
\eneq
\end{prop}

\begin{proof}
This is a slight refinement of Lemma 2.1 of \cite{SW}. 
It should be noted, however, that finitely generated assumption on \CA s in ${\cal S}$ is not really needed. 
As in the proof of Lemma 3.2 of \cite{W1}, let $\gamma<\ep_{i+1}/14.$ 
Choose a subset ${\cal F}'\subset B_i$ such that, for each $x\in {\cal F},$ 
there exists $x'\in {\cal F}'$ such that $\|x-x'\|<\ep_i.$ 
 For any $\gamma>0,$  there exists  $\vartheta>0$  and a finite subset ${\cal G}\subset B_i$ with the following
 property: if $E$ is another \CA\, $p\in E$ is a projection and $\phi: B_i\to E$ is a \hm\,  satisfying 
 $\|p\phi(b)-\phi(b)p\|<\vartheta$ for all $b\in {\cal G},$ then there exists a unital \hm\, 
 ${\bar \phi}: B_i\to pEp$ such that $\|{\bar \phi}(x')-p\phi(x')p\|<\gamma$ for all $x'\in {\cal F}'.$ 
 
 Then, as in the proof of Lemma 2.1 of \cite{SW}, one choose $B_{i+1}':=\varrho(B_i)\oplus F.$
 By property (5) in \ref{DS}, there exists a unital \SCA\, $B_{i+1}''\in \varrho(B_i)$  with $B_{i+1}''\in {\cal S}$
 such that
 \beq\label{LSW-2}
 {\rm dist}(\varrho(x'), B_{i+1}'')<\gamma/2\rforal x'\in {\cal F}'.
 \eneq
 Put $B_{i+1}=B_{i+1}''\oplus F.$
 Then the rest of proof remains the same.
 
\end{proof}

We will also use the following characterization for \CA s in $TA{\cal S}.$

\begin{thm}\label{LW} {\rm (Theorem 2.2 of \cite{W13})}
Let ${\cal S}$ be a class of unital \CA s in \ref{DS} and let 
$A$ be a unital separable simple \CA\, with $T(A)\not=\emptyset$ and with finite 
nuclear dimension.  Suppose that there are two sequences of \morp s:  $\phi_n: A\to B_n$ and 
$h_n: B_n\to A$ satisfying the following:

{\rm (i)} $B_n\in {\cal S}$ for each  $n\in \N,$

{\rm (ii)} $h_n$ is an  embedding for  each $n\in \N,$

{\rm (iii)} $\lim_{n\to\infty}\|\phi_n(a)\phi_n(b)-\phi_n(ab)\|=0\rforal a, b\in A$ and

{\rm (iv)} $\lim_{n\to\infty}\sup_{\tau\in T(A)}|\tau\circ h_n\circ \phi_n(a)-\tau(a)|=0$ for each $a\in A.$

Then $A\otimes Q$ is $TA{\cal S}.$ 

\end{thm}

\begin{lem}\label{L1}
Let $C$ and $A$ be two unital separable simple amenable \CA s.
Suppose that $C\otimes U$ has finite  nuclear dimension and $gTR(A\otimes U)\le 1$ for some UHF-alegbra $U$ of infinite type. 
Suppose also that there exists an order isomorphism 
$$
\kappa_0: (K_0(C\otimes U), K_0(C\otimes U)_+, [1_{C\otimes U}])
\to (K_0(A\otimes U), K_0(A\otimes U)_+, [1_{A\otimes U}]),
$$ 
an affine homeomorphism 
$\lambda: T(A\otimes U)\to T(C\otimes U)$ such that $(\kappa_0, \lambda)$ is compatible and there exists 
a sequence of unital \morp s $\psi_n: C\otimes U\to A\otimes U$  such that
\beq\label{L1-1}
\lim_{n\to\infty}\sup_{\tau\in T(A}|\tau(\psi_n(a))-\lambda(\tau)(a)|=0\rforal a\in C\otimes U\andeqn\\
\lim_{n\to\infty}\|\psi_n(ab)-\psi_n(a)\psi_n(b)\|=0\rforal a, b\in C\otimes U.
\eneq
Then $gTR(C\otimes U)\le 1.$ 
\end{lem}

\begin{proof}
By \cite{LS}, it suffices to show that $gTR(A\otimes Q)\le 1.$
Since $U\otimes Q\cong Q$ and the assumption holds when we replace $U$ by $Q,$ 
we may assume that $U=Q.$  
Let $C_1=C\otimes U$ and $A_1=A\otimes U.$ 
For unital separable simple \CA s with finite nuclear dimension,  we use a characterization 
for $TA{\cal C}_0$ in  \cite{W13}.
We will show that there exists a sequence of  unital \CA s 
$B_n\in {\cal C}_0$  
such that and a sequence of monomorphisms $h_n: B_n\to C_1$ such that 
\beq\label{L1-3}
\lim_{n\to\infty}\sup_{\tau\in T(C_1)}|\tau(h_n\circ \psi_n(c))-\tau(c)|=0\rforal c\in C_1.
\eneq
Then, by   \ref{LW} (Theorem 2.2 in \cite{W13}), $gTR(C_1)\le 1.$

By the assumption, we have 
\beq\label{L1-4}
r_{A_1}(\tau)(\kappa_0(x))=r_{C_1}(\lambda(\tau))(x)\tforal x\in K_0(C_1).
\eneq

Let $\{{\cal G}_n\}$ be an increasing sequence of finite subsets of $A_1$ such that
$\cup_{n=1}^{\infty}{\cal G}_n$ is dense in $A_1.$  
For each $n,$ there exists a \SCA\, $B_n\subset A_1$ with $1_{B_n}=p_n$ such that 
$B_n\in {\cal C}_0$ and a \morp\, $L_n: A_1\to B_n$ such that
\beq\label{L1-6}
\|p_nx-xp_n\|<1/2^{n+3}\tforal x\in {\cal G}_n,\\\label{MTT1-5+1}
\|L_n(x)-p_nxp_n\|<1/2^{n+3}\tforal x\in {\cal G}_n\andeqn\\\label{MTT1-5+2}
\tau(1-p_n)<{1\over{2^{n+3}\max\{\|g\|+1: g\in {\cal G}_n\}}}\rforal \tau\in T(A_1).
\eneq
In particular, 
\beq\label{L1-10+}
\lim_{n\to\infty}\|L_n(ab)-L_n(a)L_n(b)\|=0\rforal a,b\in A_1.
\eneq

It follows from Theorem 2.5 of \cite{BT} that $\text{Cu}(A_1)=V(A_1)\sqcup {\rm LAff}_+(T(A_1)).$
Since both $A_1$ and $C_1$ have stable rank one, 
the map $\Gamma: \text{Cu}(A_1)\to \text{Cu}(C_1)$
defined by 
\beq\label{L1-7}
\Gamma([p])=\kappa_0^{-1}([p])\andeqn \Gamma(f)(\lambda(\tau))=f(\tau)
\eneq
for all projections $p\in A_1\otimes {\cal K},$ $f\in {\rm LAff}_+(T(A_1))$ and 
$\tau\in T(A_1)$ is a semigroup isomorphism.
Moreover $\Gamma$  is order preserving, preserves the suprema and preserves 
the  relation of compact containment. 
Denote by $j_n: B_n\to A_1$ the embedding
and denote by $j_m^{cu}: \text{Cu}(B_n)\to \text{Cu}(A_1)$ the morphism induced by $j_m.$
It follows from the  existence part of a result in \cite{Rob} that 
there is, for each $n,$ a \hm\, $h_n: B_n\to C_1$ such that
$h_n^{cu}=\Gamma\circ j_m^{cu}.$ Since $j_m$ is the embedding and 
$\Gamma$ is an isomorphism, $h_n^{cu}$ does not vanish. It follows 
that $h_n$ is an embedding. 

Put $\psi_n=L_n\circ \psi_n'.$ 
Let $\{{\cal F}_n\}$ be an increasing sequence of finite subsets of $C_1$ whose union is dense in $C_1.$
By  passing to a subsequence,  \wtlog\, we may assume 
that, for each $n,$ 
\beq\label{L1-8}
&&\hspace{-0.4in}\sup_{\tau\in T(A_1)}|\tau\circ \psi_n(c)-\lambda(\tau)(c)|<{1\over{2^{n+1}(\max\{\|c\|+1: c\in {\cal F}_n\})}} \rforal c\in {\cal F}_n\andeqn\\
&&{\rm dist}(\psi_n'(c), {\cal G}_n)<1/2^{n+2}\rforal c\in {\cal F}_n.
\eneq
Let $g_c\in {\cal G}_n$ such that 
\beq\label{MTT1-7+}
\|\psi_n'(c)-g_c\|<1/2^{n+3}\rforal c\in {\cal F}_n.
\eneq
We then estimate that
\beq\label{L1-9}
\lambda(\tau)(h_n\circ \psi_n(c))&=&\Gamma\circ j_m^{cu}(\psi_n(c))(\tau)=\widehat{j_m(\psi_n(c))}(\tau)\\.
&=& \tau(\psi_n(c))\approx_{1/2^{n+3}} \tau(L_n(g_c))\approx_{1/2^{n+3}}\tau(p_ng_cp_n)\\
&\approx_{1/2^{n+3}}&\tau(g_c)\approx_{1/2^{n+3}}\tau(\psi_n'(c))\\
& \approx_{1/2^{n+3}}&\lambda(\tau)(c)
\eneq
for all $\tau\in T(A_1),$  where  the first approximation follows from (\ref{MTT1-7+}), 
the second follows from (\ref{MTT1-5+1}), the third follows from (\ref{MTT1-5+2}), the forth follows from (\ref{MTT1-7+})
 and  the last one follows (\ref{L1-8}).
In other words,
\beq\label{MTT1-9}
\sup_{t\in T(C_1)}|t(h_n\circ \psi_n(c))-t(c)|<1/2^n\rforal c\in C_1.
\eneq

Then, by applying   Theorem \ref{LW},  $gTR(C_1)\le 1.$ 

\end{proof}

 \section{Uniqueness and existence theorems}
 
 The following follows from a result in \cite{GLN}. 
  
 \begin{thm}\label{MUN2}
Let $X$ be a compact metric space, let $C=C(X),$ let $A_1\in {\cal A}_1,$ let
$U$  be a UHF-algebra of infinite type and
let  $A=A_1\otimes U.$  Suppose that
$\phi_1, \phi_2: C\to A$ are two unital monomorphisms.
Suppose also that
\beq\label{MUN2-1}
[\phi_1]=[\phi_2]\,\,\,{\rm in}\,\,\, KL(C,A)\\
(\phi_1)_T=(\phi_2)_T\andeqn \phi_1^{\ddag}=\phi_2^{\ddag}.
\eneq
Then $\phi_1$ and $\phi_2$ are approximately unitarily equivalent.
\end{thm}

\begin{proof}
This follows  immediately from  Theorem 12.7 of \cite{GLN}. 
Define, for each $a\in A_+^{\bf 1}\setminus \{0\},$
\beq\label{MUN2-2}
\Delta(\hat{a})=\inf\{\tau(\phi_1(a)):\tau\in T(A)\}/2.
\eneq
It is clear that $\Delta: A_+^{q, {\bf 1}}\setminus \{0\}\to (0,1)$ is non-decreasing map.
We 
\beq\label{MUN2-3}
\tau(\phi_2(a))=\tau(\phi_1(a))\ge \Delta(\hat{a})\rforal a\in A_+^{\bf 1}\setminus \{0\}.
\eneq
Then the theorem  follows   Theorem 12.7  of \cite{GLN}.

\end{proof}

\begin{df}\label{dfchi}
 {\rm 
 Let $A$ be a unital separable stably finite \CA\, with $T(A)\not=\emptyset.$
 There is a splitting exact sequence with the splitting map $J_c^A:$
 \beq\label{dfchi-1}
 0\to {\rm Aff}(T(A))/\overline{\rho_A(K_0(A))}\to U_{\infty}(A)/CU_{\infty}(A)\stackrel{\pi_{K_1}^A}{\to} K_1(A)\to 0.
 \eneq
 (see \cite{KT}). 
In particular, $\pi_{K_1}^A\circ J_c^A={\rm id}_{K_1(A)}.$ In what follows, {\it for each such unital \CA\, $A,$ 
we fix one $J_c^A.$ }
 
For a fixed pair of unital stable finite \CA s $C$ and $A$ with $T(C)\not=\emptyset$ and 
$T(A)\not=\emptyset,$  we fix $J_c^C$ and $J_c^A.$
Suppose that $\phi: C\to A$ is a unital \hm.  
Define  $\phi^{\rho}: K_1(C)\to \Aff(T(A))/\overline{\rho_A(K_0(A))}$ by 
\beq\label{Dphirho}
\phi^{\rho}=({\rm id}-J_c^A\circ \pi_{K_1}^A)\circ \phi^{\ddag}\circ J_c^C.
\eneq
Note that $\pi_{K_1}^A\circ \phi^{\rho}=0.$ 
Moreover
\beq\label{Dphirho-2}
\pi_{K_1}^A\circ \phi^{\ddag}\circ J_c^C=\phi_{*1}.
\eneq
Suppose that $\psi: C\to A$ is another unital \hm\, such that
$\psi_{*1}=\phi_{*1}$ and $\phi^{\rho}=\psi^{\rho}.$ 
Then, by (\ref{Dphirho-2}),
\beq\label{Dphirho-3}
\phi^{\ddag}\circ J_c^C=\psi^{\ddag}\circ J_c^C.
\eneq
Thus, if $\psi_T=\phi_T,$ then $\phi^{\ddag}=\psi^{\ddag}.$ 

Suppose that 
$\kappa\in KL_e(C,A)^{++},$ $\lambda: T(A)\to T(C)$ is an affine continuous map and
$\gamma: U(M_{\infty}(C))/CU(M_{\infty}(C))\to U(M_{\infty}(A))/CU(M_{\infty}(A))$ 
is a continuous \hm. We say $(\kappa, \lambda, \gamma)$ is compatible, 
if $(\kappa|_{K_0(C)}, \lambda)$ is compatible,  
$\gamma|_{\Aff(T(C))/\overline{\rho_A(K_0(C))}}=\overline{\lambda_{\sharp}},$ 
where $\overline{\lambda}_{\sharp}$ is defined in \ref{DEll}, and $\kappa|_{K_1(C)}=\pi_{K_1}^A\circ \gamma\circ J_c^C$
(which is independent of choice of $J_c^C$).

}
\end{df}

Therefore we also have the following:

\begin{cor}\label{UniCtoAC}
Let $X$ be a compact metric space, let $C=C(X),$ let $A_1\in {\cal A}_1,$ let
$U$  be a UHF-algebra of infinite type and
let  $A=A_1\otimes U.$  Suppose that
$\phi_1, \phi_2: C\to A$ are two unital monomorphisms.
Suppose also that
\beq\label{MUN2-1}
[\phi_1]=[\phi_2]\,\,\,{\rm in}\,\,\, KL(C,A)\\
(\phi_1)_T=(\phi_2)_T\andeqn \phi_1^{\rho}=\phi_2^{\rho}.
\eneq
Then $\phi_1$ and $\phi_2$ are approximately unitarily equivalent.
\end{cor}


 The following is well-known.
 
 \begin{prop}\label{pullback}
 Let $X$ be a compact metric space and let $X_n$ be a sequence of polyhedrons such that
 $C(X)=\lim_{n\to\infty} (C(X_n), s_n).$ Then, for any $\ep>0$ and any finite subset 
 ${\cal F}\subset C(X),$ there exists an integer $k_1\ge 1$  such that, 
 for any $n\ge k$ there is a  unital $\ep$-${\cal F}$-multiplicative \morp s $L_n: C(X)\to C(X_n)$
 such that 
 \beq\label{pullback-1}
\|s_{n, \infty}\circ L_n(f)-f\|<\ep \tforal f\in {\cal F}.
 \eneq
 \end{prop}
 
 \begin{proof}
 For any $\ep>0$ and any finite subset ${\cal F}\subset C(X),$ there 
 is an integer $k_1\ge 1$ such that 
 ${\rm dist}(f, s_{n, \infty}(C(X_n))<\ep/2$ for all $n\ge k_1.$
 To simplify the notation, without loss of generality, we may assume that
 $\|f\|\le 1$ for all $f\in {\cal F}.$  
 Put $B_n=s_{n, \infty}(C(X_n)).$ 
 Since $C(X)$ is amenable, there exists a unital \morp\, 
 $\Phi_n: C(X)\to B_n$ such that
 \beq\label{pullback-2}
 \|\Phi_n(f)-f\|<\ep/4\rforal f\in {\cal F}.
 \eneq
 
 Since $s_{n, \infty}(C(X_n))$ is amenable, there exists a unital \morp\, $\Psi_n: s_{n, \infty}(C(X_n))\to C(X_n)$ such that
 \beq\label{pullback-3}
 s_{n, \infty}\circ \Psi_n(f)=f
 \eneq
 for all $f\in s_{n, \infty}(C(X_n)).$
 For all $f, g\in {\cal F},$ since 
 $$
 s_{n, \infty}\circ \Psi_n(fg)=fg=s_{n, \infty}\circ \Psi_n(f)\cdot s_{n, \infty}\circ \Psi_n(g)=s_{n, \infty}(\Psi_n(f)\Psi_n(g)),
 $$
 there is an integer $k_2>n$ such that
 \beq\label{pullback-10}
 \|s_{n, m}\circ \Psi_n(fg)-s_{n,m}\circ \Psi_n(f)s_{n,m}\circ \Psi_n(g)\|<\ep/2
 \eneq
 for all $m\ge k_2$ and $f, g\in {\cal F}.$
 Put $k=k_1k_2.$  For $m\ge k,$
 define $L_m=s_{k_1+1, m}\circ \Psi_{k_1+1}\circ \Phi_{k_1+1}: C(X)\to C(X_m).$ $m=1,2,....$
 It follows from (\ref{pullback-2}) and (\ref{pullback-3}) that 
 $L_m$ is $\ep$-${\cal F}$-multiplicative. Moreover, by (\ref{pullback-2}) and (\ref{pullback-3}),
 $$
 \|s_{m, \infty}\circ L_m(f)-f\|<\ep\rforal f\in {\cal F}.
 $$
 
 \end{proof}

 \begin{lem}\label{EXpre}
 Let $X$ be a compact metric space,  let $A\in {\cal N}_1,$ let $U$ be a UHF-algebra of infinite type
 and let $B=A\otimes U.$ 
 Suppose that $\kappa\in KL_e(C(X), B)^{++},$ $\lambda: T(B)\to T_{\bf f}(C(X))$ is a 
 continuous affine map and suppose that $\gamma: U_{\infty}(C(X))/CU_{\infty}(C(X))\to U_{\infty}(A)/CU_{\infty}(B)$ 
 is a continuous \hm\, such that
 $(\kappa, \lambda, \gamma)$ is compatible. 
 Then there exists a  sequence of unital \morp s $\phi_n: C(X)\to B$ such 
 that, for any finite subset ${\cal P}\subset \underline{K}(C(X)),$ 
 \beq\label{ExCpre-1}
 &&[\phi_n]|_{\cal P}=\kappa|_{\cal P} \,\,\,{\it for\,\, all \,\,\, sufficiently \,\,\, large}\,\,\, n,\\
 &&\lim_{n\to\infty}\|\phi_n(fg)-\phi_n(f)\phi_n(g)\|=0\tforal f,\, g\in C(X),\\
 &&\lim_{n\to\infty}\max_{\tau\in T(A)}|\tau\circ \phi_n(f)-\lambda(\tau)(f)|=0\tforal f\in C(X)_{s.a.}\tand\\\label{ExCpre-1+}
 &&\lim_{n\to\infty}{\rm dist}(\overline{\langle \phi_n(v)\rangle}, \gamma({\bar v}))=0 \tforal 
 { unitaries}\,\,\,v\in M_m(C(X)), \,\,m=1,2,....
 \eneq
 \end{lem}
 (See 14.5 and 2.20 of \cite{Lnmem4} for notations in (\ref{ExCpre-1+}).)
 \begin{proof}
 We first note that,  for any compact metric space $Y,$ $T_{\bf f}(C(Y))\not=\emptyset.$ 
This is well known, but to see this, let  $\{y_n\}$ be a dense sequence in $Y.$ 
Define 
\beq\label{EXpre-1}
t(f)=\sum_{n=1}^{\infty} f(y_n)/2^n\rforal f\in C(Y).
\eneq
 It is clear that $t$ gives a faithful tracial state on $C(Y).$ 
 
 We will write $C(X)=\lim_{n\to\infty}(C(X_n), \iota_n),$ where each $X_n$ is  a polyhedron and 
 $\iota_{n, \infty}: C(X_n)\to C(X)$ is injective (see Satz 1. p. 229 of \cite{Freu} and see also 
 \cite{Mar}).   Put $C=C(X)$ and $C_n=C(X_n),$ $n=1,2,....$ 

Put $\kappa_n=\kappa\circ [\iota_{n, \infty}]$ and $\gamma_n=\gamma\circ \iota^{\ddag}.$
Define $\lambda_n: T(A)\to T_{\bf f}(C_n)$  by $\lambda_n(\tau)(f)=\lambda(\tau)(\iota_{n, \infty}(f))$
for all $f\in C_n$ and $\tau\in T(A),$ $n=1,2,....$  Note that $(\kappa_n, \lambda_n, \gamma_n)$ is compatible 
since $(\kappa, \lambda, \gamma)$ is compatible. Note also that since 
$\kappa\in KL_e(C, A)^{++}$ and $\iota_{n, \infty}$ is injective, $\kappa_n\in KL_e(C_n, A)^{++}.$
It follows from Theorem 21.14 of \cite{GLN} that there exists a unital monomorphism 
$h_n: C_n\to A$ such that
\beq\label{Expre-2}
[h_n]=\kappa_n,\,\,\, (h_n)_T=\lambda_n\andeqn h_n^{\ddag}=\gamma_n,\,\,\,n=1,2,....
\eneq

By applying \ref{pullback}, we may assume that there are \morp s $L_n: C(X)\to C(X_n)$ such that
\beq\label{Expre-3}
\lim_{n\to\infty}\|L_n(f)L_n(g)-L_n(fg)\|=0\rforal f, g\in C\andeqn\\
\lim_{n\to\infty}\|\iota_{n, \infty}\circ L_n(f)-f\|=0\rforal f\in C.
\eneq
Define $\phi_n=h_n\circ L_n,$ $n=1,2....$ One easily checks that $\{\phi_n\}$ meets the requirements.

 \end{proof}

 \begin{NN}\label{Strictp}
 
 {\rm
 Let $A$ be a unital separable stably finite simple \CA\, with $T(A)\not=\emptyset$ and let 
 $X$ be a compact metric space. Suppose that $\lambda: T(A)\to T_{\bf f}(C(X))$ is an affine continuous map.
 Define $\lambda_\sharp: C(X)_{s.a.}\to \Aff(T(A))$ by 
 $\lambda_{\sharp} (f)(\tau)=\lambda(\tau)(f)$ for all $f\in C(X)_{s.a.}.$ Since $\lambda(\tau)\in T_{\bf f}(C(X)),$
 $\lambda_\sharp$ is strictly positive in the sense that $\lambda_\sharp(f)>0$ for all $f\in C(X)_+\setminus \{0\}.$ 
 Put 
 \beq\label{Strictp-1}
 \Delta_0(\hat{f})=\inf\{\lambda_\sharp(f)(\tau): \tau\in T(A)\}
 \eneq
 for all $f\in C(X)_+^{\bf 1}.$ Note that, since $T(A)$ is compact, $1\ge \Delta_0(\hat{f})>0$ for 
 all $f\in C(X)_+^{\bf 1}.$ Thus $\Delta=3\Delta_0/4$ gives 
 an non-decreasing function from $C(X)_+^{q, \bf 1}\setminus\{0\}$ to $(0,1).$ 
 
 For any finite subset 
 ${\cal H}_0\subset C(X)_+^{q, {\bf 1}},$ there exists $\sigma>0$ and  a finite subset 
 ${\cal H}\subset C(X)_{s.a.}$ satisfying the following:
 for any unital \morp\, $\Phi: C(X)\to A$ 
 \beq\label{Strictp-2}
 \tau\circ \Phi(g)\ge \Delta(\hat{g})\rforal g\in {\cal H}_0,
 \eneq
 provided that
 \beq\label{Strictp-3}
 \max_{\tau\in T(A)}|\tau\circ \Phi(f)-\lambda(\tau)(f)|<\sigma\tforal f\in {\cal H}.
 \eneq
 
 }
 
 \end{NN}

 \begin{thm}\label{ExC}
 Let $X$ be a compact metric space, let $A\in {\cal N}_1$ and let $B=A\otimes U,$ where 
 $U$ is a UHF-algebra of infinite type.
 Suppose that $\kappa\in KL_e(C(X), B)^{++},$ $\lambda: T(A)\to T_{\bf f}(C(X))$ is a 
 continuous affine map and suppose that  $\chi: K_1(C(X))\to \Aff(T(B))/{\overline{\rho_A(K_0(B))}}$ is a \hm\, such that
 $(\kappa, \lambda)$ is compatible. 
 Then there exists a unital monomorphism $h: C(X)\to B$ such 
 that
 \beq\label{ExC-1}
 [h]=\kappa,\,\,\,h^{\rho}=\chi\andeqn  h_T=\lambda.
 \eneq
  \end{thm}
 (Here $J_c^{C(X)}$ and $J_c^B$ are fixed.)
 \begin{proof}
 Let $\Delta$ be induced by $\lambda$ as defined in \ref{Strictp}. 
 Define $\gamma: U(M_{\infty}(C(X))/CU(M_{\infty}(C(X))\to U(M_{\infty}(B))/CU(M_{\infty}(B))$ as follows:
 \beq\label{ExC-1gamma}
 \gamma|_{\Aff(T(C(X)))/\overline{\rho_{C(X)}(K_0(C(X)))}}=\overline{\lambda_{\sharp}}\andeqn\\
 \gamma|_{J_c^{C(X)}(K_0(C(X)))}=\chi\circ \pi_{K_1}^{C(X)}+ J_c^A\circ \kappa\circ \pi_{K_1}^{C(X)}.
 \eneq
 Then $(\kappa, \lambda,\gamma)$ is compatible. 
 
 Let $\{{\cal F}_n\}$ be an increasing sequence of finite subsets of $C(X)$ such that
 its union is dense in $C(X).$ Let $\{\ep_n\}$ be a decreasing sequence of positive numbers such that
 $\sum_{n=1}^{\infty} \ep_n<\infty.$ 
 We will apply  12.7 of \cite{GLN} with $C=C(X).$  Let ${\cal H}_{1,n}\subset C(X)_+^{q, {\bf 1}}\setminus \{0\}$
 (in place of ${\cal H}_1$) be a finite subset, $\sigma_{1,n}>0$ (in place of $\gamma_1$) , 
 $\sigma_{2,n}>0,$  $\dt_n>0$ (in place of $\dt$), ${\cal G}_n\subset C(X)$ (in place of ${\cal G}$) be a finite subset,
 ${\cal P}_n\subset \underline{K}(C(X))$(in place of ${\cal G}$) be a finite subset, 
 ${\cal H}_{2, n}\subset C(X)_{s.a.}$ (in place of ${\cal H}_2$) be a finite subset, 
 ${\cal U}_n\subset U_{\infty}(C(X))/CU_{\infty}(C(X))$(in place of ${\cal U}$) for which 
 $[{\cal U}_n]\subset {\cal P}_n$ be a finite subset required by 12.7 of \cite{GLN} for 
 $\ep_n$ and ${\cal F}_n,$ $n=1,2,....$ 
 We may assume that $\{{\cal H}_{1,n}\}$ and  $\{{\cal P}_n\}$ are increasing.
 By applying \ref{EXpre}, one obtains a sequence of unital \morp s $\phi_n: C(X)\to B$
 such that $\phi_n$ is $\dt_n$-${\cal G}_n$-multiplicative, 
 \beq\label{EXC-6}
 [\phi_n]|_{{\cal P}_n}&=&[\kappa]|_{{\cal P}_n},\\
 \max_{\tau\in T(A)}|\tau\circ \phi_n(f)-\lambda(\tau)(f)|&<&\sigma_{1,n}\rforal f\in {\cal H}_{2,n}\andeqn\\
 {\rm dist}(\overline{\langle \phi_n(v) \rangle}, \gamma({\bar v}))&<&\sigma_{2,n}\rforal v\in {\cal U}_n,
 \eneq
 $n=1,2,....$
 By \ref{Strictp}, we may also assume that
 \beq\label{EXC-6+}
 \tau\circ \psi_n(g)\ge \Delta(\hat{g})\rforal g\in {\cal H}_{1,n}.
 \eneq
 It follows from 12.7 of \cite{GLN} that there exists a sequence of unitaries $u_n\in B$ such that
 \beq\label{ExC-7}
 \|{\rm Ad}\, u_n\circ \phi_{n+1}(f)-\phi_n(f)\|<\ep_n\tforal f\in {\cal F}_n,\,\,\, n=1,2,....
 \eneq
 Define $\psi_1=\phi_1,$ $\psi_2={\rm Ad}\, u_1\circ \phi_1,$
 $\psi_3={\rm Ad}\, u_2\circ \psi_2,$ and $\psi_{n+1}={\rm Ad}\, u_n\circ \psi_n,$ $n=1,2,....$
 Then, by (\ref{ExC-7}), 
 \beq\label{ExC-8}
 \|\psi_{n+1}(f)-\psi_n(f)\|<\ep_n\rforal f\in {\cal F}_n,\,\,\, n=1,2,....
 \eneq
 Since $\sum_{n=1}^{\infty}\ep_n<\infty,$   $\{{\cal F}_n\}$ is an increasing sequence and 
 $\cup_{n=1}^{\infty}{\cal F}_n$ is dense in $C(X),$ it is easy to see that
 $\{\psi_n(f)\}$ is Caucy for every $f\in C(X).$ Let $h(f)=\lim_{n\to\infty}\psi_n(f)$ for $f\in C(X).$
 It is clear that $h: C(X)\to B$ is a unital \hm. Moreover, 
 \beq\label{ExC-9}
 [h]=\kappa\,\,\,{\rm in}\,\,\, KL(C(X), B),\,\,\, h_T=\lambda \andeqn h^{\ddag}=\gamma.
 \eneq
Since $\gamma(\tau)\in T_{\bf f}(C(X))$ for each $\tau\in T(B),$ $h$ is also injective.

 \end{proof}

\section{The main results}


We begin with  the following lemma.

\begin{lem}\label{Traceunif}
Let $C$ be a unital separable amenable \CA\, and let $\af\in Aut(C)$ be such that
$\tau(c)=\tau(E(c))$ for all $c\in C\rtimes_\af\Z$ and for all $\tau\in T(C\rtimes_\af\Z),$ where 
$E$ is the canonical conditional expectation.
Let  $A$ be a unital  \CA\, with $T(A)\not=\emptyset.$ 
Suppose that $\lambda: T(A)\to T(C\rtimes_\af\Z)$ is a surjective continuous affine map. 
Suppose $\phi: C\to A$ is a unital monomrophism and $\psi_n: C\rtimes_\af\Z\to A$ is  a 
 \morp, $n=1,2,...,$ such that
\beq\label{Traceun-1}
\lim_{n\to\infty}\|\psi_n(a)\psi_n(b)-\psi_n(ab)\|&=&0\tforal\, a, b\in C\rtimes_\af\Z, \\\label{Traceun-1+1}
\lim_{n\to\infty}\|\psi_n(c)-\phi(c)\|&=&0\tforal c\in C
\andeqn\\
\tau\circ \phi(a)&=&\lambda(\tau)(a)\tforal a\in C\andeqn \tau\in T(A).
\eneq
Then, for any $\ep>0$ and any finite subset set ${\cal F}\subset C\rtimes_\af\Z,$ there exists $N\ge 1$ such that
\beq\label{Traceun-2}
\sup_{\tau\in T(A)} |\tau\circ \psi_n(a)-\lambda(\tau)(a)|<\ep\tforal a\in {\cal F}
\eneq
and for all $n\ge N.$
\end{lem}

\begin{proof}
We first show that, for any $\tau\in T(A),$   
\beq\label{Trun-0-n1}
\lim_{n\to\infty}|\tau\circ \psi_n(c)-\lambda(\tau)(c)|=0\tforal c\in C\rtimes_\af\Z.
\eneq
Fix $\tau\in T(A).$ 
If (\ref{Trun-0-n1}) fails, then there exists at least one $c\not=0$ in $C\rtimes_\af\Z,$ one   $\tau\in T(C\rtimes_\af\Z)$ and a subsequence 
$\{n_k\}$ such that
\beq\label{Trun-0-n1+}
\lim\inf_n|\tau\circ \psi_{n_k}(c)-\lambda(\tau)(c)|=\eta>0.
\eneq
Since the state space  of $C\rtimes_\af\Z$ is weak*-compact, one can choose a limit point $t$
of $\{\tau\circ \psi_{n_k}\}.$ Then there exists 
a sequence $\{n_k'\}\subset \{n_k\}$ such that 
\beq\label{Trun-0-n2}
\lim_{k\to\infty}|\tau\circ \psi_{n_k'}(c)-t(c)|=0\tforal c\in C.
\eneq
By (\ref{Traceun-1}), $t$ is a tracial state. 
Let $E: C\to C(X)$ be the canonical conditional expectation.
Then, by the assumption,  $t(c)=t(E(c))$ for all $c\in C.$  By combing with (\ref{Trun-0-n2}),
\beq\label{Trun-0-n3}
\lim_{n\to\infty}|\tau\circ \psi_{n_k'}(c)-\tau\circ \psi_{n_k'}(E(c))|=0\rforal c\in C.
\eneq
However, by (\ref{Traceun-1+1}) and by (\ref{Trun-0-n3})
\beq\label{Trun-0-n4}
\lim_{n\to\infty}|\tau\circ \psi_{n_k'}(c)-\lambda(\tau)(c)|&=&
\lim_{n\to\infty}|\tau\circ \psi_{n_k'}(c)-\lambda(\tau)(E(c))|\\
&=&\lim_{n\to\infty}|\tau\circ \psi_{n_k'}(c)-\tau\circ \phi(E(c))|\\
&=&\lim_{n\to\infty}|\tau\circ \psi_{n_k'}(E(c))-\tau\circ \phi(E(c))|=0
\eneq
for all $c\in C.$ 
This contradicts with (\ref{Trun-0-n1+}). So the claim is proved. 

Suppose that the lemma is not true. 
There exists $\ep_0>0,$ a finite subset ${\cal F},$ a sequence of  tracial states $\{\tau_k\}\subset  T(B),$ and an increasing sequence $\{n_k\}$ of integers such that
\beq\label{Traceun-3}
\max_{a\in {\cal F}}|\tau_k\circ \psi_{n_k}(a)-\lambda(\tau_k)(a)|\ge \ep_0
\eneq
for all $k\ge 1.$ 
We will again use the fact that the state space of 
$C\rtimes_\af \Z$ is weak*-compact. Let $t_0$ be a limit point of $\{\tau_k\circ \psi_{n_k}\}$ in $S(C\rtimes_\af\Z).$
It follows from (\ref{Traceun-1}) that $t_0\in T(C\rtimes_\af\Z).$ 
For each $c\in C\rtimes_\af\Z,$ 
\beq\label{Traceun-4}
\lim_{k\to\infty}|\tau_k\circ \psi_{n_k}(c)-t_0(c)|=0\rforal c\in C\rtimes_\af\Z.
\eneq
However, $t_0(c)=t_0(E(c))$ for any $c\in C\rtimes_\af\Z.$ 
This implies that
\beq\label{Traceun-5}
\lim_{k\to\infty}\tau_k\circ \psi_{n_k}(E(c))=t_0(E(c))=t_0(c)\rforal c\in C\rtimes_\af\Z.
\eneq
It follows that
\beq\label{Traceun-6}
\lim_{k\to\infty}|t_k\circ \psi_{n_k}(c)-t_k\circ \psi_{n_k}(E(c))|=0\rforal c\in C\rtimes_\af\Z.
\eneq
On the other hand, by (\ref{Traceun-1+1}),
\beq\label{Traceun-7}
\lim_{k\to\infty}|t_k\circ \psi_{n_k}(E(c))-\lambda(\tau_k)(c)|&=&\lim_{k\to\infty}|t_k\circ \phi(E(c))-\lambda(\tau_k)(E(c))|=0
\eneq
for all $c\in C\rtimes_\af\Z.$
In particular, for any $a\in {\cal F},$ 
\beq\label{Traceun-8}
\lim_{k\to\infty}|t_k\circ \psi_{n_k}(a)-\lambda(\tau_k)(a)|=0.
\eneq
This contradicts with (\ref{Traceun-3}).
\end{proof}

We now prove the following:

\begin{thm}\label{MTD}
Let $X$ be an infinite compact metric space and let $\af: X\to X$ be a minimal homeomorphism.
Then $gTR((C(X)\rtimes_\af\Z)\otimes U)\le 1$ for any UHF-alegbra $U$ of infinite type.
\end{thm}

\begin{proof}
By \cite{LS}, it suffices to show that $gTR((C(X)\rtimes_\af\Z)\otimes Q))\le 1.$ 
Let $C=C(X)\rtimes_\af \Z$ and $C_1=C\otimes Q.$ 
We will use the result in \cite{ENSTW} that $C_1$ has finite nuclear dimension (when 
$X$ has finite covering dimension, it was proved in \cite{TW}). 


Note that $C_1$ is a  unital separable simple amenable ${\cal Z}$-stable \CA. 
By the Range Theorem (Theorem 13.41) in \cite{GLN}, there is a unital separable simple \CA\, $A$ in UCT class with $gTR(A)\le 1$
such that
\beq\label{Existence-0-3}
\hspace{-0.2in}(K_0(C_1), K_0(C_1)_+, [1_{C_1}], K_1(C_1), T(C_1), r_{C_1})=(K_0(A), K_0(A)_+, [1_A], K_1(A), T(A), r_A).
\eneq
Since $C_1\cong C_1\otimes Q,$ we may assume that $A\cong A\otimes Q.$ 

Let $\kappa\in KL_e(C_1, A)^{++}$ which gives the part of the above identification:
$$
(K_0(C_1), K_0(C_1)_+, [1_{C_1}], K_1(C_1))=(K_0(A), K_0(A)_+, [1_A], K_1(A)).
$$
Let $\lambda: T(A)\to T(C_1)$ be an affine homeomorphism  given by (\ref{Existence-0-3}) which is compatible with $\kappa.$
Put $\kappa_0=\kappa|_{K_0(C_1)}$ and $\kappa_1=\kappa|_{K_1(C_1)}.$
 Let $\imath_T: T(C_1)\to T_{\bf f}(C(X))$ be defined 
by $\imath_T(\tau)(f)=\tau(\imath(f))$ for all $\tau\in T(C_1)$ and for all $f\in C(X),$ 
where $\imath: C(X)\to C$ is the embedding.
By \ref{ExC}, there exists a unital monomorphism $\phi': C(X)\to A$ such that
\beq\label{Existence-0-4}
[\phi']=\kappa\circ [\imath],\,\,\, \phi'_T=\imath_T\circ \lambda\andeqn (\phi')^{\rho}=0.
\eneq
Consider $\psi=\phi'\circ \af: C(X)\to A.$ 
Then $[\imath\circ \af]=[\imath],$ $(\imath\circ \af)_T=\imath_T.$ It follows that
\beq\label{Existence-0-5}
[\psi]=[\phi'], \psi_T=\phi'_T.
\eneq
Note, since $\psi^{\rho}=\phi^{\rho}\circ \af_{*1},$ 
\beq\label{Existence-0-6}
\psi^{\rho}=(\phi')^{\rho}.
\eneq
It follows from \ref{UniCtoAC} that there exists a sequence of unitaries $\{u_n\}\subset A$ such that
\beq\label{Existence-0-7}
\lim_{n\to\infty}\|u_n^*\phi'(f)u_n-\phi'\circ \af(f)\|=0\rforal f\in C(X).
\eneq 
Let $C_0$ be the  subalgebra of $C$ whose elements have the form 
$\sum_{i=-k}^k f_iu^i,$ where $f_i\in C(X)$ and $u\in C$ is a unitary which implement 
the action $\af,$ i.e., $u^*fu=f\circ \af$ for all $f\in C(X).$ 
Define a linear map $L_n: C_0\to A$ by 
\beq\label{Existence-0-8}
L_n(\sum_{i=-k}^k f_i u^i)=\sum_{i=-k}^k f_i u_n^i.
\eneq
Let $L: C_0\to l^{\infty}(A)$ be defined by $L(c)=\{L_n(c)\}$ and let $\pi: l^{\infty}(A)\to l^{\infty}(A)/c_0(A)$
be the quotient map. Then 
$\pi\circ L: C_0\to l^{\infty}(A)/c_0(A)$ is a unital $*$-\hm. In particular, it is a covariant representation 
of $(C(X), \af).$ Thus $\pi\circ L$ gives a unital \hm\, $\Phi: C\to l^{\infty}(A)/c_0(A)$ such that
$\Phi|_{C_0}=\pi\circ L.$ Since $C$ is amenable, there exist a \morp\, $\Lambda: C\to l^{\infty}(A)$ such 
that $\pi\circ \Lambda=\Phi.$
Let $\pi_n: l^{\infty}(A)\to A$ be the projection to the $n$-th coordinate. Put $\phi_n'=\pi_n\circ \Lambda.$
Then $\phi_n': C\to A$ is a \morp. Moreover,  since $\pi\circ \Lambda=\Phi,$ we have 
\beq\label{Existence-0-9}
\lim_{n\to\infty}\|L_n(c)-\phi_n'(c)\|=0\rforal c\in C.
\eneq
In particular, 
\beq\label{Existence-0-10}
&&\lim_{n\to\infty}\|\phi'_n(a)\phi'_n(b)-\phi'_n(ab)\|=0\rforal a, b\in C\andeqn\\
&&\lim_{n\to\infty}\|\phi'_n(a)-\phi(a)\|=0\rforal a\in C(X).
\eneq
Note that  $\phi_T=\imath_T\circ \lambda,$ i.e.,
\beq\label{Existence-0-11}
\tau\circ \phi(a)=\lambda(\tau)(a)\rforal a\in C.
\eneq
It follows from \ref{Traceunif} that 
\beq\label{Existence-0-12}
\lim_{n\to\infty}\sup_{\tau\in T(A)}|\tau\circ \phi'_n(c)-\lambda(\tau)(c)|=0\rforal c\in C.
\eneq
We then define $\phi_n: C_1\to A\otimes U\cong A$ by $\phi_n(c\otimes a)=\phi_n'(c)\otimes a$ for all $c\in C$ and $a\in U.$
Then
\beq\label{Existence-0-12}
\lim_{n\to\infty}\|\phi_n(a)\phi_n(b)-\phi_n(ab)\|=0\rforal a,b\in C_1,\\
\lim_{n\to\infty}\sup_{\tau\in T(A)}|\tau\circ \phi_n(a)-\lambda(\tau)(a)|=0\rforal a\in C_1.
\eneq
Since $C_1$ has finite nuclear dimension, by \ref{L1}, the above implies that $gTR(C_1)\le 1.$ 

\end{proof}

\begin{cor}\label{CC1}
Let $X$  be a compact metric space, let $\af: X\to X$ be a minimal homeomorphism and  let $C=C(X)\rtimes_\af \Z.$
Then $C\otimes {\cal Z}\in {\cal N}_1^{\cal Z}.$ 
\end{cor}

{\bf The proof of Theorem \ref{MTB} and Theorem \ref{MTA}}:

 It is proved in \cite{EN} that, when $(X, \af)$ is minimal dynamical system with mean dimension zero, 
 $C(X)\rtimes_\af\Z$ is ${\cal Z}$-stable. Thus Theorem \ref{MTB} follows immediately from 
 \ref{CC1}. 
 When $X$ has finite dimension, every minimal dynamical system $(X, \af)$ has mean dimension zero.
 So Theorem \ref{MTA} follows.
 
 \vspace{0.2in}

 {\bf The proof of \ref{MCC}}: 
 
 Let $C=C(X)\rtimes_\af\Z.$ By \ref{MTB}, $C\in {\cal N}_1^{\cal Z}.$ 
 It follows from the Range Theorem (Theorem 13.41) in \cite{GLN} that there exists a unital simple inductive limit  $A$ 
 of sub-homogenous \CA s described in \ref{DMtorus} such that
 ${\rm Ell}(A)\cong {\rm Ell}(C).$ By the isomorphism theorem (Theorem 29.4) in \cite{GLN}, $C\cong A.$ 
 
 \vspace{0.2in}
 
 {\bf The proof of Theorem \ref{Trtr1}}:
 As mentioned earlier, this can be proved without using \cite{GLN}. The  ``only if"  part follows from \cite{LNjfa}.
 For ``if" part,  let $C=C(X)\rtimes_\af \Z.$ 
By \cite{LNjfa} again, there is  a  unital simple amenable \CA\, $A$ which satisfies the UCT and 
 $A\otimes Q$ has tracial rank at most one such that 
 ${\rm Ell}(A)\cong {\rm Ell}(C).$  Using this $A,$ exactly the same proof of \ref{MTD} shows 
 that $TR(C\otimes Q)\le 1.$ 
 
 The following is course a consequence of of \ref{MTD}. But it also a colliery of \ref{Trtr1}:
 \begin{cor}\label{CCC}
 Let $X$ be a infinite  compact metric space and let $\af: X\to X$ be a minimal homeomorphism.
 Suppose that $K_0(C(X)\rtimes_\af\Z)$ has a unique state. 
 Then $TR((C(X)\rtimes_\af\Z)\otimes Q)\le 1.$  
 \end{cor}
 
 \begin{proof}
 Put $C=C(X)\rtimes_\af\Z$ and $C_1=C\otimes Q.$  Then $S_{[1_{C_1}]}(K_0(C_1))$ has a single point. 
 Therefore it is a Choquet simplex.  In particular, by 5.10 of \cite{LNjfa}, $K_0(C_1)$ is a dimension group.
 It is also clear that all extremal points of $T(C_1)$ maps to the extremal point of $S_{[1_{C_1}]}(K_0(C_1)).$
 So \ref{Trtr1} applies.
 \end{proof}
 
 The above certain applies to all cases that studied in \cite{Lncjm15}  (see Theorem 6.1 and 6.2 of \cite{Lncjm15}). It also 
 applies to all connected $X$ with torsion $K_0(C(X)).$ 
 
 \begin{thm}\label{ANT}
 Let $C\in {\cal N}_1$ be a unital separable simple \CA\, and let $\af\in Aut(C)$ such 
 that $\tau(a)=\tau(E(a))$ for all $a\in C\rtimes_\af\Z$ and $\tau\in T(C\rtimes_\af\Z).$ 
 Suppose that $(C\rtimes_\af\Z)\otimes U$ has finite nuclear dimension for some UHF-algebra $U$ of infinite type.
 Then $C\rtimes_\af\Z\in {\cal N}_1.$ 
  \end{thm}

\begin{proof}
The proof is almost identical to that of \ref{MTD} except that we have to use a different 
existence and uniqueness theorems. 
Let $C_1=C\rtimes_\af\Z$ and $C_2=C_1\otimes Q.$  We will show that $gTR(C_2)\le 1.$ 

It follows from  4.4 of \cite{Kjfa96} that $\af^k$ is strongly outer for all integer $k\not=0.$ Therefore $C_1$ is simple (see Theorem 3.1 of \cite{Kcmp81}). 
Hence $C_2$ is a  unital separable simple amenable ${\cal Z}$-stable \CA. 
By the Range Theorem (Theorem 13.41)  in \cite{GLN}, there is a unital separable simple \CA\, $A$ in UCT class with $gTR(A)\le 1$
such that
\beq\label{ANT-3}
\hspace{-0.2in}(K_0(C_2), K_0(C_2)_+, [1_{C_2}], K_1(C_2), T(C_2), r_{C_2})=(K_0(A), K_0(A)_+, [1_A], K_1(A), T(A), r_A).
\eneq
Since $C_2\cong C_2\otimes Q,$ we may assume that $A\cong A\otimes Q.$ 

Let $\kappa\in KL_e(C_2, A)^{++}$ which gives the part of the above identification:
$$
(K_0(C_2), K_0(C_2)_+, [1_{C_2}], K_1(C_2))=(K_0(A), K_0(A)_+, [1_A], K_1(A)).
$$
Let $\lambda: T(A)\to T(C_2)$ be an affine homeomorphism  above which is compatible with $\kappa.$
Put $\kappa_0=\kappa|_{K_0(C_2)}$ and $\kappa_1=\kappa|_{K_1(C_2)}.$
 Let $\imath_T: T(C_2)\to T(C\otimes Q)$ be defined 
by $\imath_T(\tau)(c)=\tau(\imath(c))$ for all $\tau\in T(C_2)$ and for all $c\in C_2,$ 
where $\imath: C\otimes Q\to C_1\otimes Q$ is the embedding.
Since $(\kappa_0, \lambda)$ is compatible, 
$\lambda$ induces an isomorphism 
$$
{\bar \lambda_\sharp}: \Aff(T(C_2))/\overline{\rho_{C_2}(K_0(C_2))}\to
\Aff(T(A))/\overline{\rho_A(K_0(A))}.
$$

Fix $J_c^{C_2}$ and $J_c^A$ as in \ref{dfchi}. Define $\gamma: U(C\otimes Q)/CU(C\otimes Q)\to U(A)/CU(A)$
as follows:
\beq\label{ANT-6}
\gamma|_{\Aff(T(C\otimes Q))/\overline{\rho_{C\otimes Q}(K_0(C\otimes Q))}}={\overline{ (\imath_T\circ \lambda)_{\sharp}}}\andeqn\\
\gamma|_{J_c^{C\otimes Q}(K_1(C\otimes Q)}=J_c^{A}\circ \kappa_1\circ \imath_{*1}\circ (\pi_{K_1(C\otimes Q)})|_{J_c^{C\otimes Q}(K_1(C\otimes Q))}.
\eneq

By  Theorem 21.9  of \cite{GLN}, there exists a unital monomorphism $\phi: C\otimes Q\to A$ such that
\beq\label{ANT-7}
[\phi]=\kappa\circ [\imath],\,\,\, \phi_T=\imath_T\circ \lambda\andeqn (\phi)^{\ddag}=\gamma.
\eneq
Define $\bt=\af\otimes {\rm id}_U: C\otimes Q\to C\otimes Q$ and consider  $\psi=\phi\circ \bt: C\otimes Q\to A.$ 
Then $[\imath\circ \af]=[\imath],$ $(\imath\circ \af)_T=\imath_T.$ It follows that
\beq\label{ANT-8}
[\psi]=[\phi], \psi_T=\phi_T.
\eneq
Note, since  $\imath_{*1}\circ \bt_{*1}=\imath_{*1},$ 
\beq\label{ANT-9}
\psi^{\ddag}|_{\Aff(T(C\otimes U))/\overline{\rho_{C\otimes U}(K_0(C\otimes U))}}&=&{\bar \lambda_\sharp}=
\phi^{\ddag}|_{\Aff(T(C\otimes U))/\overline{\rho_{C\otimes U}(K_0(C\otimes U))}}\andeqn\\
\psi^{\ddag}|_{J_c^{C\otimes U}(K_1(C\otimes U))}&=&\phi^{\ddag}\circ \bt^{\ddag}|_{J_c^{C\otimes U}(K_1(C\otimes U))}\\
&=&J_c^{A}\circ \kappa_1\circ \imath_{*1}\circ \bt_{*1}\circ  \pi_{K_1}^{C\otimes U}|_{J_c^{C\otimes U}(K_1(C\otimes U))}\\
&=&J_c^{A}\circ \kappa_1\circ \imath_{*1}\circ \pi_{K_1}^{C\otimes U}|_{J_c^{C\otimes U}(K_1(C\otimes U))}\\
&=&\phi^{\ddag}|_{J_c^{C\otimes U}(K_1(C\otimes U))}.
\eneq
It follows that $\phi^{\ddag}=\psi^{\ddag}.$ 

By  the uniqueness theorem 
(Theorem 12.11) of \cite{GLN},  that there exists a sequence of unitaries $\{u_n\}\subset A$ such that
\beq\label{ANT-10}
\lim_{n\to\infty}\|u_n^*\phi'(f)u_n-\phi'\circ \af(f)\|=0\rforal f\in C\otimes U.
\eneq 
Let $C_0$ be a subalgebra of $C$ whose elements have the form 
$\sum_{i=-k}^k f_iu^i,$ where $f_k\in C\otimes U$ and $u\in C\otimes U\rtimes_\bt\Z$ is a unitary which implement 
the action $\bt,$ i.e., $u^*fu=f\circ \bt$ for all $f\in C\otimes U.$ 
Define a linear map $L_n: C_0\to A$ by 
\beq\label{ANT-11}
L_n(\sum_{i=-k}^k f_i u^i)=\sum_{i=-k}^k f_i u_n^i.
\eneq
Since $C\in {\cal N}_1,$ $C\otimes U$ is in the so-called bootstrap class. Therefore so is $C_2\cong (C\otimes U)\rtimes_\bt\Z.$ In particular, $C_2$ is amenable. 
The same argument used in the proof of \ref{MTD} shows that 
there exists a sequence of \morp\, $\phi_n: C_2\to A$ such that 
\beq\label{ANT-12}
&&\lim_{n\to\infty}\|\phi_n(a)\phi_n(b)-\phi_n(ab)\|=0\rforal a, b\in C_2\andeqn\\
&&\lim_{n\to\infty}\|\phi_n(a)-\phi(a)\|=0\rforal a\in C\otimes U.
\eneq
Note that  $\phi_T=\imath_T\circ \lambda,$ i.e.,
\beq\label{Existence-0-11}
\tau\circ \phi(a)=\lambda(\tau)(a)\rforal a\in C\otimes U.
\eneq
It follows from \ref{Traceunif} that 
\beq\label{Existence-0-12}
\lim_{n\to\infty}\sup_{\tau\in T(A)}|\tau\circ \phi_n(c)-\lambda(\tau)(c)|=0\rforal c\in C.
\eneq
Since $C_2$ has finite nuclear dimension, by \ref{L1}, $gTR(C_2)\le 1.$
In other words, $C\rtimes_\af\Z\in {\cal N}_1.$ 
\end{proof}

\begin{rem}\label{RCC}
Let $C$ be a unital simple separable \CA\, and $\af\in Aut(C).$
By a refined argument of Kishimoto, as in Remark 2.8 of \cite{MSI}, if $\af$ has weak Rokhlin property,
then $\tau(au^k)=0$ for any $a\in C$ and any integer $k\not=0.$ Thus $\tau(a)=\tau(E(a))$ for all $a\in C\rtimes_\af\Z$ and all $\tau\in T(C\rtimes_\af\Z),$ where $E: C\rtimes_\af\Z\to C$ is the  canonical conditional expectation. On the other hand, if $\tau(a)=\tau(E(a))$ for all $a\in C\rtimes_\af\Z$ and all $\tau\in T(C\rtimes_\af\Z),$ by the proof of 4.4 of \cite{Kjfa96}, 
the $\Z$ action induced by $\af$ is strongly outer (see also Remark 2.8 of \cite{MSI}). 
Thus, in \ref{ANT}, if in addition, $C\in {\cal N}_1^{\cal Z},$ then by Corollary 4.10 of \cite{MSII},  $C\rtimes_\af \Z$ 
is ${\cal Z}$-stable.  In light of this, it seems that the condition 
$\tau(a)=\tau(E(a))$ for all $a\in C\rtimes_\af\Z$ is a reasonable replacement for some version of weak Rokhlin 
property.
\end{rem}



hlin@uoregon.edu

\end{document}